\documentclass[a4paper,11pt,english]{smfart}
\usepackage{amsfonts}
\usepackage{amsmath} 
\usepackage{hyperref} 
\usepackage{latexsym}
\usepackage{array}
\usepackage{amssymb}
\usepackage{smfthm}
\theoremstyle{plain}

\newcommand{\R}{  \mathbb{R}   }
\newcommand{\E}{  \mathbb{E}   }

\newcommand{\eps}{\varepsilon}

\newcommand{\e}{  \text{e}   }
\newcommand{\C}{  \mathbb{C}   }
\newcommand{\Z}{  \mathbb{Z}   }
\newcommand{\N}{  \mathbb{N}   }

\newcommand{\T}{  \mathbb{T}   }

\newcommand{\wt}{  \widetilde   }

\newcommand{\im}{  \text{Im}\;   }
\newcommand{\re}{  \text{Re}\;   }
\newcommand{\om}{  \omega   }
\newcommand{\p}{  \partial   }
\newcommand{\dis}{  \displaystyle   }
\newcommand{\ov}{  \overline  }
\renewcommand{\a}{  \alpha   }

\newcommand{\s}{  \sigma   }

\renewcommand{\>}{  \rangle   }
\renewcommand{\phi}{  \varphi   }
\numberwithin{equation}{section}
\author{ Laurent Thomann }
\address{Universit\'e de Nantes, Laboratoire de Math\'ematiques J. Leray, UMR CNRS 6629\\
2, rue de la Houssini\`ere \\
F-44322 Nantes Cedex 03, France. }
\email{laurent.thomann@univ-nantes.fr}
\urladdr{http://www.math.sciences.univ-nantes.fr/$\sim$thomann/}

\author{ Nikolay Tzvetkov}
\address{D\'epartement de Math\'ematiques, Universit\'e de Cergy-Pontoise, Site Saint-Martin, 95302
Cergy-Pontoise Cedex, France.}
\email{nikolay.tzvetkov@u-cergy.fr}

\title[Gibbs measure for the periodic DNLS]{Gibbs measure for the periodic derivative nonlinear Schr\"odinger equation} 

\begin{document}
\frontmatter
 \begin{abstract}
 In this paper we construct a Gibbs measure for the derivative Schr\"odinger equation on the circle. The construction uses some renormalisations of Gaussian series and Wiener chaos estimates, ideas which have already been used by the second author in a work on the Benjamin-Ono equation.
\end{abstract}

\subjclass{35BXX ;  37K05 ; 37L50 ; 35Q55}
\keywords{Nonlinear Schr\"odinger equation,  random data, Gibbs measure}
\thanks{The authors were supported in part by the  grant ANR-07-BLAN-0250.}

\maketitle
\mainmatter


\section{Introduction}

Denote by $\T=\R/2\pi\Z$  the circle. The purpose of this work is to construct a Gibbs measure associated to the   derivative nonlinear Schr\"odinger equation
\begin{equation}\label{DNLS}
\left\{
\begin{aligned}
&i\partial_t u+\p_{x}^{2} u   = i\partial_{x}\big(|u|^{2}u\big),\;\;
(t,x)\in\R\times \mathbb{T},\\
&u(0,x)= u_{0}(x).
\end{aligned}
\right.
\end{equation}
Many recent results  (see the end of Section \ref{Sect1.2}) show that a Gibbs measure is an efficient tool to construct global rough solutions of nonlinear dispersive equations. This is the main motivation of this paper: we hope that our result combined with a local existence theory for \eqref{DNLS} (e.g. a result like Gr\"unrock-Herr \cite{GrunHerr}) on the support of the measure will give a global existence result for irregular initial conditions. A second motivation is the fact that an invariant measure  is an object  which fits well in the study of recurrence properties given by the Poincar\'e theorem, of the flow of  \eqref{DNLS}.\\[2pt] 
For $f\in L^{2}(\T)$, denote by
$\dis \int_{\T}f(x)\text{d}x=\frac{1}{2\pi}\int_{0}^{{2\pi}}f(x)\text{d}x$.
The following quantities are conserved (at least formally) by the flow of the equation\\
$\bullet$ The mass 
\begin{equation*}
M(u(t))=\|u(t)\|_{{L^{2}}}=\|u_{0}\|_{{L^{2}}}=M(u_{0}).
\end{equation*} 
$\bullet$ The energy
\begin{eqnarray*}
H(u(t))&=&\int_{\T}|\partial_{x}u|^{2}\text{d}x+\frac32\im\int_{\T}|u|^{2}u\,\partial_{x}\ov{u}\text{d}x+\frac12\int_{\T}|u|^{6}\text{d}x\\\nonumber
&=&\int_{\T}|\partial_{x}u|^{2}\text{d}x-\frac34\im\int_{\T}\ov{u}^{2}\,\partial_{x}(u^{2})\text{d}x+\frac12\int_{\T}|u|^{6}\text{d}x\\\nonumber
&=&\int_{\T}|\partial_{x}u|^{2}\text{d}x+\frac34i\int_{\T}\ov{u}^{2}\,\partial_{x}(u^{2})\text{d}x+\frac12\int_{\T}|u|^{6}\text{d}x\\\nonumber
&=&H(u_{0}).
\end{eqnarray*} 
The conservation of the energy can be seen by a direct computation (see also the appendix of this paper.)\\
Notice that the momentum
\begin{eqnarray*}
P(u(t))&=&\frac12\int_{\T}|u|^{4}\text{d}x+i \int_{\T}\ov{u}\,\partial_{x}u\text{d}x\\
&=&\frac12\int_{\T}|u|^{4}\text{d}x-\im \int_{\T}\ov{u}\,\partial_{x}u\text{d}x=P(u_{0}),
\end{eqnarray*} 
is also formally conserved by \eqref{DNLS}. Indeed it is the Hamiltonian of \eqref{DNLS} associated to a symplectic structure involving $\partial_{x}$ (see \cite{KP}). However, we won't use this fact here. Instead, our measure will be deduced from a Hamiltonian formulation based on $H$ of a transformed form of \eqref{DNLS}. \\

   Let us define the complex vector space $\dis {E_N={\rm span}\Big( (  \e^{inx})_{-N\leq n\leq N}\Big)}$. Then we introduce  the spectral projector $\Pi_{N}$ on $E_{N}$ by 
\begin{equation}\label{proj}
\Pi_{N}\Big(\sum_{n\in \Z}c_n \e^{inx}\Big)=\sum_{n=-N}^{N}c_n \e^{inx}\,.
\end{equation}~\\
\noindent Let $({\Omega, \mathcal{F},{\bf p}})$ be a probability space and $\big(g_{n}(\om)\big)_{n\in \Z}$ a sequence of independent complex normalised gaussians, $g_{n}\in \mathcal{N}_{\C}(0,1)$. We can write  
\begin{equation}\label{def.gn}
g_{n}(\om)=\frac1{\sqrt{2}}\big(h_{n}(\om)+il_{n}(\om)\big),
\end{equation}
where$\big(h_{n}(\om)\big)_{n\in \Z}, \,\big(l_{n}(\om)\big)_{n\in \Z}$ are independent standard real Gaussians  $\mathcal{N}_{\R}(0,1)$.~\\

\subsection{Definition of the  measure for \eqref{DNLS}}~\\

\noindent 
In the sequel we will use the notation $\<n\>=\sqrt{n^{2}+1}$. \\[5pt]
 Now write $c_{n}=a_{n}+ib_{n}$. \\For $N\geq 1$, consider the probability measure on $\R^{2(2N+1)}$ defined by
\begin{equation}\label{def.mu}
\text{d}\mu_N=d_{N}\prod_{n=-N}^{N}e^{-\<n\>^{2}(a_n^2+b_n^2)}\text{d}a_n \text{d}b_n,
\end{equation}
 where $d_N$ is such that
\begin{multline}
\frac{1}{d_{N}}=\prod_{ n=-N}^{N}\int_{\R^2}e^{-\<n\>^{2}(a_n^2+b_n^2)}\text{d}a_n\text{d}b_n\\
=\pi^{2N+1}\Big(\,\prod_{n=-N}^{N}\frac{1}{\<n\>}\, \Big)^{2}=  \pi^{2N+1}\Big(\,\prod_{n=1}^{N}\frac{1}{\<n\>} \,\Big)^{4}  \,.
\end{multline}
The measure $\mu_N$ defines a measure on $E_N$ via the map
$$
(a_n,b_n)_{n=-N}^{N}\longmapsto \sum_{n=-N}^{N}(a_n+ib_n)e^{inx},
$$
which will still be denoted by $\mu_N$. 
 Then $\mu_N$ may be seen as the distribution of the $E_N$ valued random
variable
\begin{equation}\label{def.phin}
\omega\longmapsto \sum_{|n|\leq N}\frac{g_{n}(\om)}{\<n\>}\e^{inx}\equiv
\varphi_{N}(\omega,x),
\end{equation}
where $(g_n)_{n=-N}^{N}$ are  Gaussians as in \eqref{def.gn}.\\[5pt]

\noindent Let $\sigma<\frac12$. Then $(\varphi_N)$ is a Cauchy sequence in $L^2\big(\Omega;\,{H}^{\sigma}(\T)\big)$ 
which defines
\begin{equation}\label{def.phi}
\phi(\om,x)=\sum_{n\in \Z}\frac{g_{n}(\om)}{\<n\>}\e^{inx},
\end{equation}
as the limit of $(\varphi_N)$. Indeed, the map
$$
\omega\longmapsto \sum_{n\in \Z}\frac{g_{n}(\om)}{\<n\>}\e^{inx},
$$
defines a (Gaussian) measure on ${ H}^{\sigma}(\T)$ which will be denoted
by $\mu$.\\[5pt]
For $u\in L^{2}(\T)$, we will write $u_{N}=\Pi_{N}u$. Now define 
\begin{equation*} 
f_{N}(u)=\im \int_{\T}\ov{u_{N}^{2}(x)}\,\partial_{x}(u_{N}^{2}(x))\text{d}x.
\end{equation*} 
\noindent Let $\kappa>0$, and let $\chi \;:\;\R\longrightarrow \R$, $0\leq \chi \leq 1$ be a continuous function with support $\text{supp}\;\chi\subset [-\kappa, \kappa]$ and so that $\chi =1$ on $[-\frac{\kappa}2, \frac{\kappa}2]$. We define the density 
\begin{equation}\label{G.N}
G_{N}(u)=\chi\big(\|u_{N}\|_{L^2(\T)}\big)
\e^{\frac{3}{4}f_{N}(u)-\frac12\int_{\T}|u_{N}(x)|^6\text{d}x},
\end{equation}
and the measure $\rho_{N}$ on $H^{\s}(\T)$ by 
\begin{equation}
\text{d}\rho_{N}(u)=G_{N}(u)\text{d}\mu(u).
\end{equation}
 
\subsection{Statement of the main result}\label{Sect1.2}~\\
Our main result which defines a formally invariant  measure for \eqref{DNLS} reads
\begin{theo}\label{thm1}
The sequence $G_{N}(u)$ defined in \eqref{G.N} converges in measure, as $N\rightarrow\infty$, with respect to the measure $\mu$.
Denote by $G(u)$ the limit of \eqref{G.N} as $N\rightarrow\infty$, and we define $\text{d}\rho(u)\equiv G(u)\text{d}\mu(u)
$.\\
Moreover,  for every $p\in [1,\infty[$, there exists $\kappa_{p}>0$ so that for all $0<\kappa\leq \kappa_{p}$,
$G(u)\in L^{p}(\text{d}\mu(u))$ and  the  sequence $G_{N}$ converges   to $G$ in $L^{p}(\text{d}\mu(u))$, as $N$ tends to
infinity. \\
\end{theo}
\begin{rema}
In particular, for any Borel set $A\subset H^{\s}(\T)$, $\lim_{N\to \infty} \rho_{N}(A)=\rho(A)$.

 It is not clear to us how to prove the convergence property, if we define $\rho_{N}$ as follows : For any Borel set $A\subset H^{\s}(\T)$,  $\rho_{N}(A)=\tilde{\rho}_{N}(A\cap E_{N})$ where $\text{d}\tilde{\rho}_{N}=G_{N}(u)\text{d} \mu_{N}(u)$. In particular, the   convergence stated in  \cite[Theorem 1]{Tzvetkov3} is not proven there. However, if we define in the context of \cite{Tzvetkov3} $\rho_{N}$ as we did here, the   convergence property holds true. In addition the measure $\rho_{N}$   defined here (see also \cite{BTT}) is more natural, since it is invariant by the truncated flow $\Phi_{N}(t)$ of equation \eqref{ODE}.
\end{rema}

One can show that by varying the cut-off ${\chi}$, the support of $\rho$ 
describes the support of $\mu$ (see Lemma \ref{4.3} below).\\[2pt]
\indent The main ideas of this paper come from the work of the second author \cite{Tzvetkov3} where a similar construction is made for the Benjamin-Ono equation using the pioneering work of Bourgain \cite{Bourgain1}. In \cite{Tzvetkov3}, one of the main difficulties is that on the support of the measure $\mu$, the $L^{2}$ norm is a.s. infinite, which is not the case in our setting, since for any $\sigma <\frac12$, $\phi(\omega)\in H^{\sigma}(\T)$, for almost all $\omega\in \Omega$. Here the difficulty is to treat the term $\dis \int_{\T}\ov{u}^{2}\,\partial_{x}(u^{2})\text{d}x $ in the conserved quantity $H$. Roughly speaking, it should be controlled by the $H^{\frac12}$ norm, but this is not enough, since $\|u\|_{H^{\frac12}(\T)}=\infty$ on the support of $\text{d}\mu$. However, we will see in Section \ref{Section2}, that we can handle this term thanks to an adapted decomposition and thanks to the integrability properties of the Gaussians. This is the main new idea in this paper. \\[5pt]
\indent The result of Theorem \ref{thm1} may be the first step to obtain almost sure global well-posedness  for \eqref{DNLS}, with initial conditions of the form \eqref{def.phi}. To reach such a  result, we will also need a suitable local existence theory on the statistical set, and prove the invariance of the measure $\text{d}\rho$ under this flow. For instance, this program was fruitful for Bourgain \cite{Bourgain2, Bourgain1} and Zhidkov \cite{Zhidkov} for NLS on the torus, Tzvetkov \cite{Tzvetkov2, Tzvetkov1} for NLS on the disc, Burq-Tzvetkov \cite{BT3} for the wave equation, Oh \cite{Oh1, Oh2} for Schr\"odinger-Benjamin-Ono and KdV  systems,  and Burq-Thomann-Tzvetkov \cite{BTT} for the one-dimensional Schr\"odinger equation.  
For the DNLS equation, we plan to pursue this issue in a subsequent work.\\
\subsection{Notations and structure of the paper}~

 \begin{enonce*}{Notations}
 In this paper $c$, $C$ denote constants the value of which may change
from line to line. These constants will always be universal, or uniformly bounded with respect to the other parameters.\\
We denote by $\Z$ (resp. $\N$) the set of the integers (resp.  non negative integers), and $\N^{*}=\N\backslash\{0\}$.\\
For $x\in \R$, we write $\<x\>=\sqrt{x^{2}+1}$. For $u\in L^{2}(\T)$, we usually write $u_{N}=\Pi_{N}u$, where $\Pi_{N}$ is the projector defined in \eqref{proj}.\\
The notation    $L^{q}$ stands for $L^{q}(\T)$ and $H^{s}=H^{s}(\T)$.
\end{enonce*}

The paper is organised as follows. In Section \ref{Section2} we give some large deviation bounds and some results on the Wiener chaos at any order. In Section \ref{Section3} we study the term of the Hamiltonian containing the derivative, and Section \ref{Section4} is devoted to the proof of Theorem \ref{thm1}.\\ In the appendix, we give the Hamiltonian formulation of the transformed form of \eqref{DNLS}.

\section{Preliminaries : some stochastic estimates}\label{Section2}~


\subsection{Large deviation estimates}~
\begin{lemm}\label{lem.ponctuel}
Let $\big(\gamma_{n}(\om)\big)_{n\in \Z}\in \mathcal{N}_{\R}(0,1)$ be a sequence of independent, normalised real Gaussians. Let $(c_{n})_{n\in \Z}$, $(d_{n})_{n\in \Z}$ be two bounded sequences of real numbers. Then there exist   $c,C>0$ so that for all $1\leq N\leq \lambda $
\begin{equation*}
{\bf p}\Big(\,\om \in \Omega \;:\; \Big| \sum_{|n_{1}|,|n_{2}|\leq N}\frac{c_{n_{1}} d_{n_{2}}}{\<n_{1}\>}\,\gamma_{n_{1}}(\om)\,\gamma_{n_{2}}(\om)  \Big|>\lambda\, \Big)\leq C e^{-c\lambda }.
\end{equation*}
\end{lemm}

\begin{proof}
We estimate 
\begin{equation*}
A_{\lambda,N}\equiv {\bf p}\Big(\,\om \in \Omega \;:\;  \sum_{|n_{1}|,|n_{2}|\leq N}\frac{c_{n_{1}} d_{n_{2}}}{\<n_{1}\>}\,\gamma_{n_{1}}(\om)\,\gamma_{n_{2}}(\om)  >\lambda\, \Big).
\end{equation*}
For all $t>0$ and all r.v. $X$ we have, by the Tchebychev inequality
\begin{equation}\label{tcheb}
{\bf p}\Big(\,\om \in \Omega \;:\; X>\lambda\, \Big)\leq \e^{-\lambda t}\,\E\big[\,\e^{tX}\,\big].
\end{equation}
Thus we obtain that  for all $\eps>0$ 
\begin{eqnarray}
A_{\lambda,N}&\leq & \e^{-t\lambda}\,  \E \Big[\,\prod_{|n_{1}|,|n_{2}|\leq N}\e^{t\frac{c_{n_{1}} d_{n_{2}}}{\<n_{1}\>}\gamma_{n_{1}}\gamma_{n_{2}}}\,\Big] 
\nonumber\\
&\leq & \e^{-t\lambda}\,  \E \Big[\,\prod_{|n_{1}|,|n_{2}|\leq N}\e^{\frac{\eps}2\frac{c^{2}_{n_{1}} }{\<n_{1}\>^{2}}\gamma^{2}_{n_{1}}+\frac1{2\eps}t^{2}d^{2}_{n_{2}}\gamma^{2}_{n_{2}}}\,\Big]\nonumber \\
&= &\e^{-t\lambda}\,  \E \Big[\,\Big(\prod_{|n_{1}|\leq N}\e^{\frac{\eps}2\frac{c^{2}_{n_{1}} }{\<n_{1}\>^{2}}\gamma^{2}_{n_{1}}}\Big)\,\Big(\prod_{|n_{2}|\leq N}\e^{\frac1{2\eps}t^{2}d^{2}_{n_{2}}\gamma^{2}_{n_{2}}}\Big)\,\Big].\label{CC}
\end{eqnarray}
Now, the  Cauchy-Schwarz inequality and the independence of the $\gamma_{n}$ give
 \begin{eqnarray}
A_{\lambda,N}&\leq &\e^{-t\lambda}\,  \E \Big[\,\prod_{|n_{1}|\leq N}\e^{\eps\frac{c^{2}_{n_{1}} }{\<n_{1}\>^{2}}\gamma^{2}_{n_{1}}}\Big]^{\frac12}\,\E\Big[ \prod_{|n_{2}|\leq N}\e^{t^{2}d^{2}_{n_{2}}\gamma^{2}_{n_{2}}/\eps}\,\Big]^{\frac12}\nonumber\\
&= &\e^{-t\lambda}\,\Bigg(\prod_{|n_{1}|\leq N} \E \Big[ \e^{\eps\frac{ c^{2}_{n_{1}} }{\<n_{1}\>^{2}}\gamma^{2}_{n_{1}}}\Big]\, \prod_{|n_{2}|\leq N}\E\Big[ \e^{t^{2}d^{2}_{n_{2}}\gamma^{2}_{n_{2}}/\eps}\,\Big]\Bigg)^{\frac12}.\label{CC*}
\end{eqnarray}
Thanks to a change of variables we can compute explicitly the expectations in the right hand side of \eqref{CC*}. In fact for $\mu<\frac12$  
\begin{equation*}
 \E \Big[\,\e^{\mu\gamma_{n}^{2}}\,\Big]=\Big(1-2\mu\Big)^{-\frac12}.
 \end{equation*}
For $0\leq x\leq \frac12$, we have the inequality $(1-x)^{-1}\leq \e^{2x}$, hence    we deduce that for $\mu<\frac14$
\begin{equation}\label{eq.int}
 \E \Big[\,\e^{\mu\gamma_{n}^{2}}\,\Big]\leq  \e^{2\mu}.
 \end{equation}
Recall that $(c_{n}), (d_{n})$ are bounded. We now fix $\eps>0$ so that for all $|n|\leq N$, $\dis \eps \frac{c^{2}_{n}}{\<n\>^{2}}\leq \frac12$. Then the   bound \eqref{eq.int} implies 
\begin{equation}\label{P1}
 \prod_{|n_{1}|\leq N} \E \Big[ \e^{\eps\frac{ c^{2}_{n_{1}} }{\<n_{1}\>^{2}}\gamma^{2}_{n_{1}}}\Big]\leq 
\exp{\Big(2\eps\,\sum_{|n_{1}|\leq N}{\frac{ c^{2}_{n_{1}} }{\<n_{1}\>^{2}}}\Big)}\leq C.
 \end{equation}
With the previous choice of $\eps >0$ and $t>0$ small enough we also have
\begin{equation}\label{P2}
 \prod_{|n_{2}|\leq N}\E\Big[ \e^{t^{2}d^{2}_{n_{2}}\gamma^{2}_{n_{2}}/\eps}\,\Big]\leq 
\exp{\Big(2t^{2}\,\sum_{|n_{2}|\leq N}{d^{2}_{n_{2}}}/\eps\Big)}\leq \e^{Ct^{2}N},
  \end{equation}
  Finally, from \eqref{P1}, \eqref{P2} and \eqref{CC*} we infer
\begin{equation*}
A_{\lambda,N}
\leq  C\e^{-t\lambda+Ct^{2}N}\leq  C\e^{-c\lambda}\,,
\end{equation*}
for some $c>0$, if $t>0$ is chosen small enough and $N\leq \lambda$.\\[3pt]
Similarly for $\lambda >0$, 
\begin{equation*}
{\bf p}\Big(\,\om \in \Omega \;:\;  \sum_{|n_{1}|,|n_{2}|\leq N}\frac{c_{n_{1}} d_{n_{2}}}{\<n_{1}\>}\,\gamma_{n_{1}}(\om)\,\gamma_{n_{2}}(\om)  <-\lambda\, \Big)\leq  C\e^{-c\lambda}\,,
\end{equation*}
and this yields the result.

\end{proof}

\begin{lemm}\label{ld}
Fix $\s<\frac12$ and $p\in [2,\infty)$. Then
\begin{multline*} 
\exists\, C>0,\exists\, c>0,\,\, \forall\, \lambda\geq 1,\, \forall\, N\geq 1,
\\
\mu\big(\,u\in { H}^{\sigma}\,:\, \|\Pi_{N}u\|_{L^{p}(\T)}>\lambda\,\big)\leq Ce^{-c\lambda^2}\,.
\end{multline*}
Moreover there exists $\beta>0$ such that
\begin{multline*} 
\exists\, C>0,\exists\, c>0,\,\, \forall\, \lambda\geq 1,\, \forall\, M\geq N\geq 1,
\\
\mu\big(\,u\in { H}^{\sigma}\,:\, \|\Pi_{M}u-\Pi_{N}u\|_{L^{p}(\T)}>\lambda\,\big)\leq 
Ce^{-cN^{\beta}\lambda^2}\,.
\end{multline*}
\end{lemm}
\begin{proof}
This result is consequence of the hypercontractivity of the Gaussian random variables : There exists $C>0$ such that for all $r\geq 2$ and $(c_{n})\in l^{2}(\N)$
\begin{equation*}
\|\sum_{n\geq 0}g_{n}(\om)\,c_{n}\|_{L^{r}(\Omega)}\leq C\sqrt{r}\Big(\sum _{n\geq 0}|c_{n}|^{2}\Big)^{\frac12}.
\end{equation*}
See e.g. \cite[Lemma 3.3]{BTT} for the details of the proof.
\end{proof}~\\

\subsection{Wiener chaos estimates}~\\

The aim of this subsection is to obtain $L^{p}(\Omega)$ bounds on Gaussian series. These are obtained thanks to the smoothing effects of the Ornstein-Uhlenbeck semi-group. The following considerations are inspired from \cite{Tzvetkov3}. See also \cite{ABCF,Tala.Ledoux} for more details on this topic. \\[5pt]
For $d\geq 1$,  denote by $L$ the operator 
\begin{equation*}
L=\Delta -x\cdot \nabla=\sum_{j=1}^{d}\big(\,\frac{\partial^{2}}{\partial x^{2}_{j}}-x_{j}\,\frac{\partial}{\partial x_{j}}\,\big).
\end{equation*} 
This operator is self adjoint on $\mathcal{K}=L^{2}(\R^{d}, \e^{-|x|^{2}/2}\text{d}x)$ with domain 
$$\mathcal{D}=\Big\{u\;:\;u(x)=\e^{|x|^{2}/4}v(x), \;\;v\in \mathcal{H}^{2}   \Big\},$$
where  $\dis \mathcal{H}^{2}=\big\{u\in L^{2}(\R^{d}),\; x^{\a}\partial^{\beta}_{x}v(x)\in L^{2}(\R^{d}), \;\;\forall\,(\alpha,\beta)\in N^{2d},\;|\alpha|+|\beta|\leq 2\big\}$.
Denote by ${\bf k}=k_{1}+\dots +k_{d}$ and  by $(P_{n})_{n\geq 0}$  the Hermite polynomials defined by 
\begin{equation*} 
P_{n}(x)=(-1)^{n}\e^{x^{2}}\frac{\text{d}^{n}}{\text{d}x^{n}}\big(\,\e^{-x^{2}}\,\big).
\end{equation*}
Then a Hilbertian basis of eigenfunctions of $L$ on $\mathcal{K}$ is given by 
\begin{equation*}
P_{\bf k}(x_{1},\dots, x_{d})=P_{k_{1}}(x_{1})\dots P_{k_{d}}(x_{d}),
\end{equation*} 
with eigenvalue $-{\bf k}=-(k_{1}+\cdots+ k_{d})$.\\
Finally define the measure $\gamma_{d}$ on $\R^{d}$ by 
\begin{equation*}
\text{d}\gamma_{d}(x)=(2\pi)^{-d/2}\e^{-|x|^{2}/2}\text{d}x.
\end{equation*} 
~\\
The next result is a direct consequence of \cite[Proposition 3.1]{Tzvetkov3}. See \cite{ABCF} for the proof.
\begin{lemm}\label{lem.smoo.L}
Let $d\geq 1$ and $k\in \N$. Assume that $\widetilde{P}_{k}$ is an eigenfunction of $L$ with eigenvalue $-k$. Then for all $p\geq 2$
\begin{equation*}
\|\widetilde{P}_{k} \|_{L^{p}(\R^{d},\text{d}\gamma_{d})} \leq (p-1)^{\frac{{ k}}2} \|\widetilde{P}_{k} \|_{L^{2}(\R^{d},\text{d}\gamma_{d})}.
\end{equation*}
\end{lemm}

 Thanks to Lemma \ref{lem.smoo.L}, we will prove the following   $L^{p}$ smoothing effect for some stochastic series.~\\
\begin{prop}[Wiener chaos]\label{prop.wiener}~\\
Let $d\geq 1$ and $c(n_{1},\dots ,n_{k}) \in \C$. Let $(g_{n})_{1\leq n\leq d}\in\mathcal{N}_{\C}(0,1)$ be complex $L^{2}$-normalised independent Gaussians.\\
 For $k\geq 1$ denote by $A(k,d)=\big\{(n_{1},\dots,n_{k})\in \{1,\dots,d\}^{k},\,n_{1}\leq \cdots \leq n_{k}\big\}$ and 
\begin{equation}\label{Sk}
S_{k}(\om)=\sum_{A(k,d) }c(n_{1},\dots ,n_{k})\,g_{n_{1}}(\om)\cdots g_{n_{k}}(\om).
\end{equation}
Then for all $d\geq 1$ and $p\geq 2$
\begin{equation*}
\|S_{k}\|_{L^{p}(\Omega)}\leq \sqrt{k+1}\,(p-1)^{\frac{k}2}\,\|S_{k}\|_{L^{2}(\Omega)}.
\end{equation*}
\end{prop}

\begin{proof}
Let $g_{n}\in\mathcal{N}_{\C}(0,1)$. Then we can write $g_{n}=\frac1{\sqrt{2}}(\gamma_{n}+i\,\widetilde{\gamma}_{n})$ with $\gamma_{n},\widetilde{\gamma}_{n} \in \mathcal{N}_{\R}(0,1)$ mutually independent Gaussians. Hence, up to a change of indexes (and with $d$ replaced with $2d$) we can assume that the random variables in \eqref{Sk} are real valued. Thus in the following we assume that $g_{n}\in\mathcal{N}_{\R}(0,1)$ and are independent.\\[2pt]
Denote by 
\begin{equation*}
\Sigma_{k}(x_{1},\dots,x_{d})=\sum_{A(k,d) }c(n_{1},\dots ,n_{k})\,x_{n_{1}}\cdots x_{n_{k}}.
\end{equation*}
Then obviously for all $p\geq 1$, 
\begin{equation}\label{egalite}
\|S_{k}\|_{L^{p}(\Omega)}=\|\Sigma_{k}\|_{L^{p}(\R^{d}, \text{d}\gamma_{d})}.
\end{equation}
Let $(n_{1},\dots ,n_{k}) \in A(k,d)$. Then we can write 
\begin{equation*}
x_{n_{1}}\cdots x_{n_{k}}=x^{p_{1}}_{m_{1}}\,\cdots \,x^{p_{l}}_{m_{l}},
\end{equation*}
where $l\leq k$, $p_{1}+\cdots+p_{l}=k$ and $n_{1}=m_{1}<\cdots<m_{l}\leq n_{k}$. Now, each monomial $x^{p_{j}}_{m_{j}}$ can be expanded on the Hermite polynomials $(P_{n})_{n\geq 0}$
\begin{equation*}
x^{p_{j}}_{m_{j}}=\sum_{k_{j}=0}^{p_{j}}\a_{j,k_{j}}\,P_{k_{j}}(x_{m_{j}}).
\end{equation*}
Therefore there exists $\beta(k_{1},\dots,k_{l})\in \C$ so that 
\begin{equation*}
x_{n_{1}}\cdots x_{n_{k}}=\sum_{j=0}^{k}\;\sum_{\substack{k_{1}+\cdots+k_{l}=j\\ 0\leq k_{i}\leq p_{i}}}\beta(k_{1},\dots,k_{l})P_{k_{1}}(x_{m_{1}})\cdots P_{k_{l}}(x_{m_{l}}),
\end{equation*}
and we have 
\begin{equation}\label{sigma_{k}}
\Sigma_{k}(x_{1},\dots,x_{d})\\
=\sum_{j=0}^{k}\widetilde{P}_{j}(x_1,\dots,x_d),
\end{equation}
where the polynomial $\widetilde{P}_{j}$ is given by 
\begin{multline*}
\widetilde{P}_{j}(x_1,\dots,x_d)\\
=\sum_{A(k,d) }\,\sum_{\substack{k_{1}+\cdots+k_{l}=j\\ 0\leq k_{i}\leq p_{i}} }c(n_{1},\dots ,n_{k})\beta(k_{1},\dots,k_{l})P_{k_{1}}(x_{m_{1}})\cdots P_{k_{l}}(x_{m_{l}}).
\end{multline*}
For $0\leq k_{i}\leq p_{i}$ so that $k_{1}+\cdots+k_{l}=j$, the polynomial $\widetilde{P}_{j}$ is an eigenfunction of $L$ with eigenvalue $-j$, hence by Lemma \ref{lem.smoo.L} we have that for all $p\geq 2$
\begin{equation*}
\|\widetilde{P}_{j}\|_{L^{p}(\R^{d},\text{d}\gamma_{d})}\leq (p-1)^{\frac{j}2}\,\|\widetilde{P}_{j}\|_{L^{2}(\R^{d},\text{d}\gamma_{d})}.
\end{equation*}
Therefore, by \eqref{sigma_{k}} and by the Cauchy-Schwarz inequality,
\begin{eqnarray*}
\|\Sigma_{k}\|_{L^{p}(\R^{d},\text{d}\gamma_{d})}&\leq &(p-1)^{\frac{k}2}\;\sum_{j=0}^{k}\|\widetilde{P}_{j}\|_{L^{2}(\R^{d},\text{d}\gamma_{d})}\\
&\leq &\sqrt{k+1}\,(p-1)^{\frac{k}2}\;\Big(\sum_{j=0}^{k}\|\widetilde{P}_{j}\|^{2}_{L^{2}(\R^{d},\text{d}\gamma_{d})}\Big)^{\frac12}\\
&\leq &\sqrt{k+1}\,(p-1)^{\frac{k}2}\;\|\Sigma_{k}\|_{L^{2}(\R^{d},\text{d}\gamma_{d})},
\end{eqnarray*}
where in the last line we used that the polynomials $\widetilde{P}_{j}$ are orthogonal. This concludes the proof  by \eqref{egalite}
\end{proof}

\noindent We will need the following lemma which is proved in \cite[Lemma 4.5]{Tzvetkov3}
\begin{lemm}\label{lem.exp}
Let $F\,:\, H^{\s}(\T)\longrightarrow \R$ be a measurable function. Assume that there exist $\a>0, N>0, k\geq 1$, and $C>0$ so that for every $p\geq 2$
\begin{equation*}
\|F\|_{L^{p}(\text{d}\mu)}\leq CN^{-\a}\,p^{\frac{k}2}.
\end{equation*}
Then there exist $\delta>0, C_{1}$ independent of $N$ and $\a$ such that
\begin{equation*}
\int_{H^{\s}(\T)}\e^{\delta N^{\frac{2\a}{k}}|F(u)|^{\frac2k}}\text{d}\mu(u)\leq C_{1}.
\end{equation*}
As a consequence, for all $\lambda>0$,
\begin{equation*}
\mu\big(\,u\in { H}^{\sigma}(\T)\,:\, |F(u)|>\lambda\,\big)\leq C_{1}e^{-\delta N^{\frac{2\a}{k}}\lambda^{\frac2k}}.
\end{equation*}
\end{lemm}
~
\section{\texorpdfstring{Study of the sequence $\big(f_{N}(u)\big)_{N\geq 1}$}{Study of the sequence fn}}\label{Section3}~

\noindent Recall that $f_{N}(u)$ is defined by 
$\dis f_{N}(u)=\im \int_{\T}\ov{u_{N}^{2}(x)}\,\partial_{x}(u_{N}^{2}(x))\text{d}x.$\\
The main result of this  section  is the following
\begin{prop}\label{prop.cauchy}
The sequence $\big(f_{N}\big)_{N\geq 1}$ is a Cauchy sequence in $L^{2}\big(H^{\s}(\T), \mathcal{B},\text{d}\mu\big)$. Indeed for all $0<\eps<\frac12$ there exists $C>0$ so that for all $M>N\geq 1$
\begin{equation}\label{L2}
\|f_{M}(u)-f_{N}(u)\|_{L^{2}\big(H^{\s}(\T), \mathcal{B},\text{d}\mu\big)}\leq \frac{C}{N^{\frac32-\eps}}.
\end{equation} 
Moreover, for all $p\geq 2$ and  $M>N\geq 1$
\begin{equation}\label{LP}
\|f_{M}(u)-f_{N}(u)\|_{L^{p}\big(H^{\s}(\T), \mathcal{B},\text{d}\mu\big)}\leq  \frac{C\,(p-1)^{2}}{N^{\frac32-\eps}}.
\end{equation} 
\end{prop}

\noindent Then a combination of  the estimate \eqref{LP} and Lemma \ref{lem.exp} yields the following large deviation estimate
\begin{coro}\label{lem.dist}
For every $\a<\frac32$, there exist $C,\delta>0$ such that for all $M>N\geq 1$ and $\lambda>0$
\begin{equation*}
\mu\big(\,u\in { H}^{\sigma}(\T)\,:\, |f_{M}(u)-f_{N}(u)|>\lambda\,\big)\leq Ce^{-\delta (N^{\a}\lambda)^{\frac12}}.
\end{equation*}
\end{coro}
Thanks to Proposition \ref{prop.cauchy}, we are able to define the limit in $L^{2}(\Omega)$ of the sequence 
$\big(f_{N}\big)_{N\geq 1}$, which will be denoted by
\begin{equation}\label{def.f} 
f(u)=\im \int_{\T}\ov{u^{2}(x)}\,\partial_{x}(u^{2}(x))\text{d}x.
\end{equation} 
This gives a sense to the r.h.s. of \eqref{def.f} for $u$ in the support of $\mu$.

Notice that Corollary \ref{lem.dist} implies in particular the convergence in measure
\begin{equation}\label{CVM}
\forall\,\eps>0,\;\; \lim_{N\to\infty}\mu\Big(u\in {\mathcal H}^{-\sigma}\,:\,
\big|f_{N}(u)-f(u)\big|>\eps)=0.
\end{equation}

 For the proof of Proposition \ref{prop.cauchy}, we have to put $\dis \int_{\T}\ov{\phi_{N}^{2}(\om)}\,\partial_{x}(\phi_{N}^{2}(\om))\text{d}x$ in a suitable form.\\
\noindent Recall the notation \eqref{def.phin}, then 
\begin{equation}\label{eq.phi}
\phi^{2}_{N}(\om)=\sum_{|n_{1}|,|n_{2}|\leq N}\frac{g_{n_{1}}(\om)\,g_{n_{2}}(\om)}{\<n_{1}\>\<n_{2}\>}\e^{i(n_{1}+n_{2})x}.
\end{equation} 
Therefore we deduce that 
\begin{equation}\label{eq.dphi}
\p_{x}\big(\phi^{2}_{N}(\om)\big)=\sum_{|m_{1}|,|m_{2}|\leq N}i(m_{1}+m_{2})\frac{g_{m_{1}}(\om)\;g_{m_{2}}(\om)}{\<m_{1}\>\<m_{2}\>}\e^{i(m_{1}+m_{2})x}.
\end{equation}
Now, by \eqref{eq.phi}, \eqref{eq.dphi} and the fact that  $\big(\e^{inx}\big)_{n\in \Z}$ is an orthonormal family in $L^{2}(\T)$ (endowed with the scalar product $\dis \<f,g\>=\int_{\T}f(x)\ov{g}(x)\text{d}x=\frac1{2\pi}\int_{0}^{2\pi}f(x)\ov{g}(x)\text{d}x$), we obtain
 
\begin{equation}\label{eq.sum}
\int_{\T}\ov{\phi_{N}^{2}(\om)}\,\partial_{x}(\phi_{N}^{2}(\om))\text{d}x=\sum_{A_{N}}i(n_{1}+n_{2})\, \frac{{g_{m_{1}}(\om)}\;{g_{m_{2}}(\om)}\;\ov{{g_{n_{1}}(\om)}}\;\ov{{g_{n_{2}}(\om)}}}{\<m_{1}\>\,\<m_{2}\>\,\<n_{1}\>\,\<n_{2}\>},
\end{equation} 
where 
\begin{equation*}
A_{N}=\{(m_{1},m_{2},n_{1},n_{2})\in \Z^{4}\;\;\text{s.t.}\;\; |m_{1}|, |m_{2}|,  |n_{1}|, |n_{2}| \leq N\;\;\text{and}\;\; m_{1}+m_{2}=n_{1}+n_{2}\}.
\end{equation*}

\noindent We now split the sum \eqref{eq.sum} in two parts, by distinguishing the cases $m_{1}=n_{1}$ and $m_{1}\neq n_{1}$ in $A_{N}$ and write 
\begin{equation*}
\int_{\T}\ov{\phi_{N}^{2}(\om)}\,\partial_{x}(\phi_{N}^{2}(\om))\text{d}x=S^{1}_{N}+S^{2}_{N},
\end{equation*} 
with 
\begin{equation}\label{sn1}
S^{1}_{N}=\sum_{B_{N}}i(n_{1}+n_{2})\, \frac{{g_{m_{1}}(\om)}\;{g_{m_{2}}(\om)}\;\ov{{g_{n_{1}}(\om)}}\;\ov{{g_{n_{2}}(\om)}}}{\<m_{1}\>\,\<m_{2}\>\,\<n_{1}\>\,\<n_{2}\>},
\end{equation} 
where $B_{N}=A_{N}\cap\{\,m_{1}=n_{1}\;\;\text{or}\;\;m_{1}=n_{2}\,\}$, and
\begin{equation}\label{sn2}
S^{2}_{N}=\sum_{\substack{A_{N}, m_{1}\neq n_{1}\\m_{1}\neq n_{2}}}i(n_{1}+n_{2})\, \frac{{g_{m_{1}}(\om)}\;{g_{m_{2}}(\om)}\;\ov{{g_{n_{1}}(\om)}}\;\ov{{g_{n_{2}}(\om)}}}{\<m_{1}\>\,\<m_{2}\>\,\<n_{1}\>\,\<n_{2}\>}.
\end{equation} 
\subsection{Study of $S^{1}_{N}$}

\begin{lemm}\label{lem.s1n}
Let $S^{1}_{N}$ be defined by \eqref{sn1}. Then there exists $C>0$ so that for all $M>N>0$,
\begin{equation*}
\|S^{1}_{M}-S^{1}_{N}\|_{L^{2}(\Omega)}\leq \frac C{N^{\frac32}}.
\end{equation*} 
\end{lemm}

\begin{proof}
Let $(m_{1},m_{2},n_{1},n_{2})\in B_{N}$.  Then  as $m_{1}+m_{2}=n_{1}+n_{2}$, we have $(m_{1},m_{2})=(n_{1},n_{2})$ or $(m_{1},m_{2})=(n_{2},n_{1})$, and deduce that 
\begin{equation*}
S^{1}_{N}=\sum_{|n_{1}|,|n_{2}|\leq N}2i(n_{1}+n_{2})\, \frac{|{g_{n_{1}}(\om)}|^{2}\;|{g_{n_{2}}(\om)}|^{2}}{\<n_{1}\>^{2}\,\<n_{2}\>^{2}}=X_{N}+Y_{N},
\end{equation*} 
where 

\begin{equation*}
X_{N}=\sum_{|n|\leq N}4in\, \frac{|{g_{n}(\om)}|^{4}}{\<n\>^{4}},
\end{equation*} 
and 
\begin{equation*}
Y_{N}=\sum_{\substack{|n_{1}|,|n_{2}|\leq N,\\ n_{1}\neq n_{2}}}2i(n_{1}+n_{2})\, \frac{|{g_{n_{1}}(\om)}|^{2}\;|{g_{n_{2}}(\om)}|^{2}}{\<n_{1}\>^{2}\,\<n_{2}\>^{2}}.
\end{equation*} 
$\spadesuit$ First we will show that there exists $C>0$ so that for all $M>N>0$,
\begin{equation}\label{XN}
\|X_{M}-X_{N}\|_{L^{2}(\Omega)}\leq \frac C{N^{2}}.
\end{equation} 
Let $M> N\geq 1$. Then 
\begin{equation*}
|X_{M}-X_{N}|^{2}=\sum_{N< |n_{1}|,|n_{2}|\leq M}16n_{1}n_{2}\, \frac{|{g_{n_{1}}(\om)}|^{4}\,|{g_{n_{2}}(\om)}|^{4}}{\<n_{1}\>^{4}\<n_{2}\>^{4}}.
\end{equation*} 
Thus 
\begin{equation*}
\|X_{M}-X_{N}\|^{2}_{L^{2}(\Omega)}\leq C\sum_{N< |n_{1}|,|n_{2}|\leq M}\, \frac{1}{\<n_{1}\>^{3}\<n_{2}\>^{3}}\leq \frac C{N^{4}},
\end{equation*} 
which proves \eqref{XN}. \\[4pt]
$\spadesuit$ To complete the proof of Lemma \ref{lem.s1n}, it remains to check that there exists $C>0$ so that for all $M>N>0$,
\begin{equation}\label{YN}
\|Y_{M}-Y_{N}\|_{L^{2}(\Omega)}\leq \frac C{N^{\frac32}}.
\end{equation} 
For $M\geq N\geq 1$ we write
\begin{eqnarray*}
Y_{N}&=&\sum_{\substack{ |n_{1}|,|n_{2}|\leq N, \\n_{1}\neq n_{2}}}i(n_{1}+n_{2})\frac{|{g_{n_{1}}(\om)}|^{2}\,|{g_{n_{2}}(\om)}|^{2}}{\<n_{1}\>^{2}\<n_{2}\>^{2}}\\
&=&Y^{1}_{N}+Y^{2}_{N}+Y^{3}_{N},
\end{eqnarray*}
with 
\begin{equation}\label{y1N}
Y^{1}_{N}=\sum_{\substack{ |n_{1}|,|n_{2}|\leq N,\\ n_{1}\neq n_{2}}}i(n_{1}+n_{2})\frac{\big(|{g_{n_{1}}(\om)}|^{2}-1\big)\,\big(|{g_{n_{2}}(\om)}|^{2}-1\big)}{\<n_{1}\>^{2}\<n_{2}\>^{2}},
\end{equation} 
\begin{equation}\label{y2N}
Y^{2}_{N}=\sum_{ \substack{|n_{1}|,|n_{2}|\leq N, \\n_{1}\neq n_{2}}}i(n_{1}+n_{2})\frac{\big(|{g_{n_{1}}(\om)}|^{2}-1\big)+\big(|{g_{n_{2}}(\om)}|^{2}-1\big)}{\<n_{1}\>^{2}\<n_{2}\>^{2}},
\end{equation} 
and
\begin{equation*}
Y^{3}_{N}=\sum_{\substack{ |n_{1}|,|n_{2}|\leq N,\\ n_{1}\neq n_{2}}}i(n_{1}+n_{2})\frac{1}{\<n_{1}\>^{2}\<n_{2}\>^{2}}.
\end{equation*}
By the symmetry $(n_{1},n_{2})\mapsto (-n_{1}, -n_{2})$, we have that $Y^{3}_{N}=0$.
For $n\in \Z$, denote by 
\begin{equation*} 
G_{n}(\om)=|g_{n}(\om)|^{2}-1.
\end{equation*}
Let $n\neq m$. Then, since $g_{n}$ and $g_{m}$ are independent and since $\E\big[|g_{n}(\om)|^{2}\big]=1$, we have 
 \begin{equation}\label{indep}
\E\big[G_{n}(\om)\,G_{m}(\om) \big]=\E\big[G_{n}(\om)\big] \E\big[G_{m}(\om) \big]=0.
\end{equation}
$\bullet$ First we analyse \eqref{y1N}. We compute 
\begin{equation*}
|Y^{1}_{M}-Y^{1}_{N}|^{2}=
\sum_{C_{M,N}}(n_{1}+n_{2})(m_{1}+m_{2})\frac{G_{m_{1}}(\om)\,G_{m_{2}}(\om)\;G_{n_{1}}(\om)\,G_{n_{2}}(\om)}{\<m_{1}\>^{2}\<m_{2}\>^{2}\<n_{1}\>^{2}\<n_{2}\>^{2}},
\end{equation*} 
where 
\begin{multline*}
C_{M,N}=
\big\{(m_{1},m_{2},n_{1},n_{2})\in \Z^{4}\;\;\text{s.t.}\;\; N<|m_{1}|, |m_{2}|,  |n_{1}|, |n_{2}| \leq M\\
\;\;\text{and}\;\; m_{1}\neq m_{2},\;\;n_{1}\neq n_{2}\big\}.
\end{multline*} 
We compute $\E\big[|Y^{1}_{M}-Y^{1}_{N}|^{2} \big]$, and thanks to \eqref{indep} we see that only the terms $(n_{1}=m_{1}\;\text{and} \;n_{2}=m_{2})$ or $(n_{1}=m_{2}\;\text{and} \;n_{2}=m_{1})$ give some contribution, hence  
\begin{eqnarray}\label{ineqC1}
\|Y^{1}_{M}-Y^{1}_{N}\|_{L^{2}(\Omega)}^{2}&\leq &C\sum_{N<|n_{1}|,|n_{2}|\leq M}\frac{(n_{1}+n_{2})^{2}}{\<n_{1}\>^{4}\<n_{2}\>^{4}}\nonumber\\
&\leq &C\sum_{N<|n_{1}|,|n_{2}|\leq M}\Big(\frac{1}{\<n_{1}\>^{2}\<n_{2}\>^{4}}+\frac{1}{\<n_{1}\>^{4}\<n_{2}\>^{2}}\Big)\leq \frac C{N^{4}}.
\end{eqnarray} 
$\bullet$ We now turn to  \eqref{y2N}. 
Similarly, we get 
\begin{equation*}
|Y^{2}_{M}-Y^{2}_{N}|^{2}=
\sum_{C_{M,N}}(n_{1}+n_{2})(m_{1}+m_{2})\frac{\big(G_{m_{1}}(\om)+G_{m_{2}}(\om)\big)\big(G_{n_{1}}(\om)+G_{n_{2}}(\om)\big)}{\<m_{1}\>^{2}\<m_{2}\>^{2}\<n_{1}\>^{2}\<n_{2}\>^{2}},
\end{equation*} 
and using the symmetries in $(n_{1},n_{2},m_{1},m_{2})$, and with \eqref{indep} we obtain 
\begin{equation}\label{ineq1}
\|Y^{2}_{M}-Y^{2}_{N}\|_{L^{2}(\Omega)}^{2}\leq C\Big|\sum_{\substack{C_{M,N},\\m_{1}=n_{1}}}\frac{(n_{1}+n_{2})(n_{1}+m_{2})}{\<m_{2}\>^{2}\<n_{1}\>^{4}\<n_{2}\>^{2}}\Big|.
\end{equation} 
We write
\begin{multline*} 
\sum_{\substack{N<|n_{2}|\leq M,\\ n_{2}\neq n_{1}}}\frac{(n_{1}+n_{2})(n_{1}+m_{2})}{\<m_{2}\>^{2}\<n_{1}\>^{4}\<n_{2}\>^{2}}=\\
\begin{aligned}
&=\Big(\sum_{N<|n_{2}|\leq M}\frac{(n_{1}+n_{2})(n_{1}+m_{2})}{\<m_{2}\>^{2}\<n_{1}\>^{4}\<n_{2}\>^{2}}\Big)-\frac{2n_{1}(n_{1}+m_{2})}{\<m_{2}\>^{2}\<n_{1}\>^{6}}.
\end{aligned}
\end{multline*}
Then, by symmetry
\begin{equation*}
\sum_{N<|n_{2}|\leq M}\frac{n_{2}(n_{1}+m_{2})}{\<m_{2}\>^{2}\<n_{1}\>^{4}\<n_{2}\>^{2}}=0,
\end{equation*}
thus 
\begin{eqnarray}\label{som1}
\sum_{\substack{N<|n_{2}|\leq M,\\ n_{2}\neq n_{1}}}\frac{(n_{1}+n_{2})(n_{1}+m_{2})}{\<m_{2}\>^{2}\<n_{1}\>^{4}\<n_{2}\>^{2}}&=&
s_{M,N}\frac{n_{1}(n_{1}+m_{2})}{\<m_{2}\>^{2}\<n_{1}\>^{4}}-\frac{2n_{1}(n_{1}+m_{2})}{\<m_{2}\>^{2}\<n_{1}\>^{6}}\nonumber \\
&=&\frac{n_{1}}{\<n_{1}\>^{4}}\big(s_{M,N}-\frac{2}{\<n_{1}\>^{2}}\big)\frac{n_{1}+m_{2}}{\<m_{2}\>^{2}},
\end{eqnarray}
with $\displaystyle s_{M,N}=\sum_{N<|n_{2}|\leq M}\frac1{\<n_{2}\>^{2}}\leq \frac {C}{N}$. 

Similarly, using that 
$
\dis \sum_{N<|m_{2}|\leq M}\frac{m_{2}}{\<m_{2}\>^{2}}=0,
$
we obtain
\begin{equation}\label{som2}
\sum_{\substack{N<|m_{2}|\leq M,\\ m_{2}\neq n_{1}}}\frac{n_{1}+m_{2}}{\<m_{2}\>^{2}}
=\Big(\sum_{N<|m_{2}|\leq M}\frac{n_{1}+m_{2}}{\<m_{2}\>^{2}}\Big)-\frac{2n_{1}}{\<n_{1}\>^{2}}=n_{1}\,s_{M,N}-\frac{2n_{1}}{\<n_{1}\>^{2}}.
\end{equation}
Then, from \eqref{som1} and \eqref{som2}, we deduce 
\begin{equation*}
\Big|\sum_{\substack{
N<|m_{2}|\leq M, \\
m_{2}\neq n_{1}
}}\sum_{\substack{N<|n_{2}|\leq M,\\ n_{2}\neq n_{1}}}\frac{(n_{1}+n_{2})(n_{1}+m_{2})}{\<m_{2}\>^{2}\<n_{1}\>^{4}\<n_{2}\>^{2}}\Big|\leq C \big(\frac1{N^{2}\<n_{1}\>^{2}}+\frac1{\<n_{1}\>^{6}}\big),
\end{equation*}
and from \eqref{ineq1}, 
\begin{equation}\label{ineqC2}
\|Y^{2}_{M}-Y^{2}_{N}\|_{L^{2}(\Omega)}^{2}\leq \frac C{N^{3}}.
\end{equation}
Finally, \eqref{ineqC1} and \eqref{ineqC2} yield the estimate \eqref{YN}.
\end{proof}
~
\subsection{Study of $S^{2}_{N}$}~\\[5pt]
We first state the elementary lemma
\begin{lemm}\label{lem.somme}
Let $n\in \Z$ and $N\geq 1$. Then for all $0<\eps\leq \frac12$
\begin{equation*}
\sum_{\substack{n_{1}\in \Z\\|n_{1}|,|n-n_{1}|\geq N }}\frac1{\<n_{1}\>^{2}\<n-n_{1}\>^{2}}\leq \frac{C}{N^{\frac32-\eps} \<n\>^{\frac32+\eps}  }.
\end{equation*} 
\end{lemm}

\begin{proof}
Let $N\geq 1$. For $\alpha>1$ we have the inequalities
$$\<n\>^{\a}\leq C\big( \<n_{1}\>^{\a}+\<n-n_{1}\>^{\a}\big),$$ 
and
$$\<n_{1}\>^{2}\<n-n_{1}\>^{2}\geq CN^{4-2\a}\<n_{1}\>^{\a}\<n-n_{1}\>^{\a}\;\;\text{for}\;\;|n_{1}|,|n-n_{1}|\geq N.$$
Now choose $\a=\frac32+\eps\leq 2$ to get 
\begin{equation}\label{sumup}
\frac{\<n\>^{\frac32+\eps}}{\<n_{1}\>^{2}\<n-n_{1}\>^{2}}\leq \frac{C}{N^{1-2\eps}} \big(\frac1{\<n_{1}\>^{\frac32+\eps}}+\frac1{\<n-n_{1}\>^{\frac32+\eps}}\big).
\end{equation} 
We sum up \eqref{sumup}, thus 
\begin{multline*}
\sum_{\substack{n_{1}\in \Z\\|n_{1}|,|n-n_{1}|\geq N }}\frac1{\<n_{1}\>^{2}\<n-n_{1}\>^{2}}\leq \\
\begin{aligned}
& \leq  \frac{C}{N^{1-2\eps} \<n\>^{\frac32+\eps}    }\sum_{\substack{n_{1}\in \Z\\|n_{1}|,|n-n_{1}|\geq N }}\big(\frac1{\<n_{1}\>^{\frac32+\eps}}+\frac1{\<n-n_{1}\>^{\frac32+\eps}}\big)\leq \frac{C}{N^{\frac32-\eps} \<n\>^{\frac32+\eps}  },
\end{aligned}
\end{multline*}
which was the claim. 
\end{proof}
~\\
We are now able to prove
\begin{lemm}\label{lem.s2n}
Let $S^{2}_{N}$ be defined by \eqref{sn2}. For all $0<\eps\leq \frac12$,  there exists $C>0$ so that for all $M>N>0$,
\begin{equation*}
\|S^{2}_{M}-S^{2}_{N}\|_{L^{2}(\Omega)}\leq \frac C{N^{\frac32-\eps}}.
\end{equation*} 
\end{lemm}

\begin{proof}
We compute 
\begin{equation*}\label{}
\big|S^{2}_{M}-{S^{2}_{N}}\big|^{2}=\sum_{D_{M,N}\times D_{M,N}}(n_{1}+n_{2})(p_{1}+p_{2})\, \frac{{g_{m_{1}}\;{g_{m_{2}}}\;\ov{{g_{n_{1}}}}\;\ov{{g_{n_{2}}}}\;{g_{p_{1}}\;{g_{p_{2}}}\;\ov{{g_{q_{1}}}}}}\;\ov{{g_{q_{2}}}}}{\<m_{1}\>\,\<m_{2}\>\,\<n_{1}\>\,\<n_{2}\>\,\<p_{1}\>\,\<p_{2}\>\,\<q_{1}\>\,\<q_{2}\>},
\end{equation*} 
where 
\begin{multline*}
D_{M,N}=
\big\{(m_{1},m_{2},n_{1},n_{2})\in \Z^{4}\;\;\text{s.t.}\;\; N<|m_{1}|, |m_{2}|,  |n_{1}|, |n_{2}| \leq M\\
\;\;\text{and}\;\; 
m_{1}+m_{2}=n_{1}+n_{2},\;\;m_{1}\neq n_{1},\;\;m_{1}\neq n_{2}\big\}.
\end{multline*} 
The expectation of each term of the previous sum vanishes, unless $\big((m_{1},m_{2})=(q_{1},q_{2})\;\;\text{or}\;\; (q_{2},q_{1})\big)$ and  $\big((n_{1},n_{2})=(p_{1},p_{2})\;\;\text{or}\;\; (p_{2},p_{1})\big)$. Hence 
\begin{equation*}
\|S^{2}_{M}-{S^{2}_{N}}\|_{L^{2}(\Omega)}^{2}\leq C\Big|\sum_{D_{M,N}} \frac{(n_{1}+n_{2})(m_{1}+m_{2})}{\<n_{1}\>^{2}\,\<n_{2}\>^{2}\,\<m_{1}\>^{2}\,\<m_{2}\>^{2}}\Big|.
\end{equation*}
Write $n=n_{1}+n_{2}=m_{1}+m_{2}$, therefore 
\begin{eqnarray*}
\|S^{2}_{M}-{S^{2}_{N}}\|_{L^{2}(\Omega)}^{2}&\leq &C\sum_{n\in \Z}\;\sum_{\substack{|n_{1}|,|n-n_{1}|> N,\\|m_{1}|,|n-m_{1}|> N}} \frac{n^{2}}{\<n_{1}\>^{2}\,\<n-n_{1}\>^{2}\,\<m_{1}\>^{2}\,\<n-m_{1}\>^{2}}\\
&=&C\sum_{n\in \Z} n^{2}\Big(\sum_{|n_{1}|,|n-n_{1}|> N}\frac{1}{\<n_{1}\>^{2}\,\<n-n_{1}\>^{2}}\Big)^{2}\\
&\leq & \frac{C}{N^{3-2\eps}}\sum_{n\in \Z}\frac{n^{2}}{ \<n\>^{3+2\eps}  }\leq \frac{C}{N^{3-2\eps}},
\end{eqnarray*}
by Lemma \ref{lem.somme}.
\end{proof}
The results of Lemmas \ref{lem.s1n} and \ref{lem.s2n} imply \eqref{L2}. 

To complete the proof of Proposition \ref{prop.cauchy}, it remains to show \eqref{LP}. But this is a direct consequence of \eqref{L2} and Proposition \ref{prop.wiener}
\\[5pt]
\noindent We are now able to define the density $G\,:\,H^{\s}(\T)\longrightarrow \R$ (with respect to the measure $\mu$) of the measure $\rho$. By \eqref{CVM} and Proposition \ref{prop.cauchy} and Lemma \ref{ld}, we have the following convergences in the $\mu$ measure :   $f_{N}(u)$ converges   to $f(u)$ and   $\|u_{N}\|_{L^{6}(\T)}$  to $\|u\|_{L^{6}(\T)}$. Then, by composition and multiplication of continuous functions, we obtain 
\begin{equation}\label{def.G}
\chi\big(\|u_{N}\|_{L^2(\T)}\big)
\e^{\frac{3}{4}f_{N}(u)-\frac12\int_{\T}|u_{N}(x)|^6\text{d}x}\longrightarrow \chi\big(\|u\|_{L^2(\T)}\big)
\e^{\frac{3}{4}f(u)-\frac{1}{2}\int_{\T}|u(x)|^6\text{d}x}\equiv G(u),
\end{equation}
in measure, with respect to the measure $\mu$. As a consequence, $G$ is measurable from $\big(H^{\s}(\T), \mathcal{B}\big)$ to $\R$.~\\

\section{\texorpdfstring{Integrability of the density of $\text{d}\rho$}{Integrability of the density of drho}}\label{Section4}~

\noindent We now state a result which will be useful for the $L^{p}$ estimates in Theorem \ref{thm1}.
\begin{prop}\label{prop1}
There exist $\kappa_{0}>0$ and $c,\,C>0$ so that for all $0<\kappa\leq \kappa_{0}$, $\lambda\geq 2$ and  $1\leq N\leq \lambda$
\begin{equation*}
\mu\big(\,u\in { H}^{\sigma}(\T)\,:\, \|\p_{x}\big(u_{N}^{2}\big)\|_{L^{\infty}(\T)}>\lambda,\;\; \|u_{N}\|_{L^{2}(\T)}\leq \kappa\,\big)\leq Ce^{-c\lambda}.
\end{equation*}
\end{prop}

\begin{proof}
We can follow the mains  lines of the proof of \cite[Proposition 4.1]{Tzvetkov3}.\\
For $j\in \{0,\cdots, [\lambda^{5}]\}$, we define the points $x_j\in \T$  by
$
\dis x_{j}=  \frac{2\pi j}{\lambda^{5}}.
$

\noindent Denote by ${\rm dist}$ the
distance on $\T$. Then by construction,  $ \dis {\rm dist}(x_{j},x_{j+1})\leq \frac{2\pi}{\lambda^{5}}$,
with $x_{[\lambda^{5}]+1}\equiv x_{0}$. We define the set $K_{\lambda}$ by
$$ K_{\lambda}\equiv
 \Big\{\,u\in H^{\sigma}(\T)\,:\, \|\p_{x}\big(u_{N}^{2}\big)\|_{L^{\infty}(\T)}\geq
\lambda ,\;\; \|u_{N}\|_{L^2(\T)}\leq \kappa\,\Big\}\,, $$
and the sets $K_{\lambda,j}$ by
$$ K_{\lambda,j}\equiv
 \Big\{\,u\in H^{\sigma}(\T)\,:\, |\p_{x}\big(u_{N}^{2}\big)(x_j)|\geq
\frac{\lambda}2 ,\;\; \|u_{N}\|_{L^2(\T)}\leq \kappa\,\Big\}\,. $$
$\bullet$ As in \cite{Tzvetkov3} we will show that 
\begin{equation}\label{union}
K_{\lambda}\subset
\bigcup_{j=0}^{[\lambda^{5}]}K_{\lambda,j}\,.
\end{equation}
Let  $u\in K_{\lambda}$, and  denote by $v_{N}=\p_{x}\big(u_{N}^{2}\big)$. Let $x^{\star}\in \T$ be such that
$$
|v_{N}(x^{\star})|=\max_{x\in \T}|v_{N}(x)|.
$$
Thus 
$
|v_{N}(x^{\star})|\geq \lambda.
$
Then there exists $j_0\in \{0,\cdots, [\lambda^{5}]\}$ such that 
\begin{equation}\label{dist}
|x^\star-x_{j_0}|\leq \frac{2\pi}{\lambda^{5}}.
\end{equation}
Then thanks to the Taylor formula, we have
\begin{eqnarray}\label{taylor}
|v_{N}(x^\star)-v_{N}(x_{j_0})|& = & \Big|\int_{x_{j_0}}^{x^\star}
\partial_{x} v_{N}(t)dt\Big|\nonumber \\
& \leq & |x^\star-x_{j_0}|^{\frac{1}{2}}\|\partial_{x} v_{N}\|_{L^2(\T)}. 
\end{eqnarray}
Now by the Sobolev embeddings we obtain the bound (with $N\leq \lambda$)
\begin{eqnarray}
\|\partial_{x} v_{N}\|_{L^2(\T)}\leq C N \| v_{N}\|_{L^2(\T)}&\leq& CN \| u_{N}\|_{L^2(\T)} \| \p_{x}u_{N}\|_{L^{\infty}(\T)}\nonumber\\
&\leq & CN^{\frac52}  \| u_{N}\|^{2}_{L^2(\T)}\leq C\lambda^{\frac52} \kappa^{2}\label{sobo}.
\end{eqnarray}
Therefore, from \eqref{dist}, \eqref{taylor} and \eqref{sobo} we deduce that for $\kappa>0$ small enough
\begin{equation*}
|v_{N}(x^\star)-u_{N}(x_{j_0})| \leq C \kappa^{2}\leq \frac12\lambda,
\end{equation*}
Thus, by the triangle inequality
$$
|v_{N}(x_{j_0})|\geq
|v_{N}(x^\star)|
-|v_{N}(x^\star)-v_{N}(x_{j_0})|\geq \lambda-\frac{1}{2}\lambda=
\frac{1}{2}\lambda\,,
$$
we can conclude that  $u\in K_{\lambda,j_0}$, which proves \eqref{union}.\\[3pt]
$\bullet$ We now estimate 
$\mu(K_{\lambda,j})$. \\
As in \cite{Tzvetkov3}, we can forget the $L^{2}$ constraint and write   
\begin{equation*}
\mu(K_{\lambda,j})\leq {\bf p}\big(\,\omega \in \Omega\,:\, \big|\p_{x}\big(\phi_{N}^{2}\big)(x_j)\big|\geq
\frac{\lambda}2\,\big).
\end{equation*}
First observe that 
\begin{multline*}
\big\{\,\omega \in \Omega\,:\, \big|\p_{x}\big(\phi_{N}^{2}\big)(x_j)\big|\geq
\frac{\lambda}2\,\big\}\subset \\
\big\{\,\omega \in \Omega\,:\, \big|\re \p_{x}\big(\phi_{N}^{2}\big)(x_j)\big|\geq
\frac{\lambda}4\,\big\} \cup \big\{\,\omega \in \Omega\,:\, \big|\im \p_{x}\big(\phi_{N}^{2}\big)(x_j)\big|\geq
\frac{\lambda}4\,\big\}.
\end{multline*}
Indeed we can describe the previous sets by the following way. Write
\begin{equation*}
\frac12 \p_{x}\big(\phi_{N}^{2}\big)(x_j)=\sum_{|n_{1}|,|n_{2}|\leq N}in_{2}\frac{g_{n_{1}}(\om)\, g_{n_{2}}(\om) }{\<n_{1}\>\<n_{2}\>}\e^{i(n_{1}+n_{2})x_{j}},
\end{equation*}
and use that $g_{n}(\om)=\frac1{\sqrt{2}}\big(h_{n}(\om)+il_{n}(\om) \big)$ where $\big(h_{n}\big)_{n\in \Z}, \big(l_{n}\big)_{n\in \Z} \in \mathcal{N}_{\R}(0,1)$ are independent. Then a straightforward computation enables us to put $\re \p_{x}\big(\phi_{N}^{2}\big)(x_j) $ and $\im \p_{x}\big(\phi_{N}^{2}\big)(x_j)$ in the form 
\begin{equation*}
\sum_{|n_{1}|,|n_{2}|\leq N}\frac{c_{n_{1}} d_{n_{2}}}{\<n_{1}\>}\,\gamma_{n_{1}}(\om)\,\gamma_{n_{2}}(\om),
\end{equation*}
with $|c_{n}|,\,|d_{n}|\leq C$ and where $\big(\gamma_{n}\big)_{n\in \Z} \in \mathcal{N}_{\R}(0,1)$ is an  independent family of real Gaussians (indeed $\gamma_{n}=h_{n}$ or $\gamma_{n}=l_{n}$). Therefore we can apply the Lemma \ref{lem.ponctuel} to get  
\begin{equation}\label{eq.exp2}
\mu(K_{\lambda,j})\leq Ce^{-c\lambda}\,.
\end{equation}
Finally by  \eqref{union} and \eqref{eq.exp2} we deduce that 
\begin{equation*}
\mu(K_{\lambda})\leq
\sum_{j=0}^{[\lambda^{5}]}
\mu(K_{\lambda,j})\leq
C\lambda^{5}e^{-c\lambda}
\leq C e^{-\frac{c}2\lambda}\,,
\end{equation*}
which was the claim.
\end{proof}

\begin{prop}\label{prop.lp}
For all $1\leq p<\infty$, there exists $\kappa_{p}>0$  so that for all $0<\kappa\leq \kappa_{p}$ there exists $C>0$ such that for every $N\geq 1$.
$$
\Big\|
\chi\big(\|u_{N}\|_{L^2(\T)}\big)
\e^{\frac{3}{4}f_{N}(u)-\frac12\int_{\T}|u_{N}(x)|^6\text{d}x}
\Big\|_{L^p(\text{d}\mu(u))}\leq C\,.
$$
\end{prop}

\begin{proof}
Here we can follow the proof of \cite[Proposition 4.9]{Tzvetkov3}. To prove the proposition, it is sufficient to show that the integral
\begin{equation}\label{integrale}
\int_{0}^{\infty}\lambda^{p-1}\mu(A_{\lambda,N})d\lambda,
\end{equation}
is convergent uniformly with respect to $N$ for $\kappa>0$ small enough and  where 
\begin{equation*}
A_{\lambda,N}=
\Big\{u\in {H}^{\sigma}\,:\,
\chi\big(\|u_{N}\|_{L^2(\T)}\big)
\e^{\frac{3}{4}f_{N}(u)-\frac12\int_{\T}|u_{N}(x)|^6\text{d}x}>\lambda
\Big\}.
\end{equation*}~
We set $N_{0}=\ln \lambda$.\\[3pt]
\noindent $\bullet$ Assume that $N_0\geq N$. \\
On the support of $\chi$, $\|u_{N}\|_{L^{2}(\T)}\leq \kappa$, thus we have
\begin{eqnarray*}
|f_{N}(u)|=\big|\im \int_{\T}\ov{u_{N}(x)}^2\,\p_{x}(u_{N}(x)^{2})\text{d}x\big| &\leq &C
\|u_{N}\|_{L^2(\T)}^{2}\|\p_{x}(u_{N}^{2})\|_{{L}^{\infty}(\T)}\\
&\leq &C\kappa^{2}\|\p_{x}(u_{N}^{2})\|_{{L}^{\infty}(\T)}\,.
\end{eqnarray*}
Then by Proposition \ref{prop1} (which can be applied, since $N\leq N_{0}=\ln \lambda\leq \frac{c_{1}}{\kappa^{2}}\ln{\lambda}$ for $\kappa>0$ small enough), we obtain 
\begin{eqnarray*}
\mu(A_{\lambda,N})&\leq &\mu \Big(u\in { H}^{\sigma}\,:\, 
|f_{N}(u)|>\frac43 \ln \lambda,\;\;\|u_{N}\|_{L^{2}(\T)}\leq \kappa\,\Big)\nonumber \\
&\leq &\mu \Big(u\in { H}^{\sigma}\,:\, 
\|\p_{x}(u_{N}^{2})\|_{{L}^{\infty}(\T)}>\frac{c_{1}}{\kappa^{2}}\ln \lambda,\;\;\|u_{N}\|_{L^{2}(\T)}\leq \kappa\,\Big)\nonumber\\
&\leq &C\e^{-\frac{c_{2}}{\kappa^{2}}\ln \lambda}=C\lambda^{-\frac{c_{2}}{\kappa^{2}}},
\end{eqnarray*}
where $c_{2}$ is independent of $\kappa$. Hence the integral \eqref{integrale} is convergent if $\kappa=\kappa_{p}>0$ is small enough.\\[3pt]
\noindent $\bullet$ Assume now  $N>N_0$. \\
Thanks to the triangle inequality
$A_{\lambda,N}\subset B_{\lambda,N}\cup C_{\lambda,N}$, where
$$
B_{\lambda,N}\equiv
\Big\{
u\in { H}^{\sigma}\,:\, 
|f_{N_{0}}(u)|>\frac12 \ln \lambda,\;\;\|u_{N}\|_{L^{2}(\T)}\leq \kappa
\Big\},
$$
and
$$
C_{\lambda,N}\equiv
\Big\{
u\in { H}^{\sigma}\,:\, 
|f_{N}(u)-f_{N_{0}}(u)|>\frac12 \ln \lambda,\;\;\|u_{N}\|_{L^{2}(\T)}\leq \kappa
\Big\}.
$$
The measure of $B_{\lambda,N}$ can be estimated exactly as we did in the analysis of
the case $N_0\geq N$.
Finally, by Corollary \ref{lem.dist}, as $N_0=\ln \lambda $, we obtain that for all $1<\a<\frac32$
$$
\mu(C_{\lambda,N})\leq Ce^{-\delta (\ln \lambda)^{\frac{1+\a}2}}
\leq C_{L}\lambda^{-L}\,,
$$
for all $L\geq 1$. This completes the proof of the proposition.
\end{proof}

 \begin{proof}[Proof of Theorem \ref{thm1}] 
 Recall \eqref{def.G}. Let  $p\in [1,+\infty)$ and choose $\kappa_{p}>0$ so that Proposition \ref{prop.lp} holds. Then there exists a subsequence $G_{N_{k}}(u)$ so that $G_{N_{k}}(u)\longrightarrow G(u)$,  $\mu$ a.s. Then by Fatou's lemma, 
  \begin{equation*}
 \int_{H^{\s}(\T)}|G(u)|^{p}\text{d}\mu(u)\leq \liminf_{k\to \infty} \int_{H^{-\s}(\T)}|G_{N_{k}}(u)|^{p}\text{d}\mu(u)\leq C,
 \end{equation*}
 thus  $G(u)\in L^{p}(\text{d}\mu(u))$.\\
 Now it remains to check the convergence in $L^{p}(\text{d}\mu(u))$ for $1\leq p<\infty$.  As in \cite{Tzvetkov3}, for $N\geq 0$ and $\eps>0$,  we introduce the set 
\begin{equation*} 
 A_{N,\eps}=\big\{u \in H^{\s}(\T)\::\:|G_{N}(u)-G(u)|\leq \eps \big\},
 \end{equation*}
 and denote by $\ov{A_{N,\eps}}$ its complement.\\
 Firstly, there exists $C>0$ so that for all $N\geq 0$, $\eps>0$
   \begin{equation*}
 \int_{A_{N,\eps}}\big|G_{N}(u)-G(u)\big|^{p}\text{d}\mu(u)\leq C\eps^{p}.
 \end{equation*}
 Secondly, by Cauchy-Schwarz, Proposition \ref{prop.lp} and as $G(u)\in L^{2p}(\text{d}\mu(u))$, 
 we obtain
   \begin{eqnarray*}
\int_{\ov{A_{N,\eps}}}\big|G_{N}(u)-G(u)\big|^{p}\text{d}\mu(u)&\leq& \|G_{N}-G\|^{p}_{L^{2p}(\text{d}\mu)}\mu(\,\ov{A_{N,\eps}}\,)^{\frac12}\\
&\leq&C \mu(\,\ov{A_{N,\eps}}\,)^{\frac12}.
 \end{eqnarray*}
By \eqref{def.G}, we deduce that for all $\eps>0$,
\begin{equation*}
\mu(\,\ov{A_{N,\eps}}\,)\longrightarrow 0, \quad N\longrightarrow +\infty,
\end{equation*}
 which yields the result. This ends the proof of Theorem \ref{thm1}. 
\end{proof}

\begin{lemm}\label{4.3}
The measure $\rho$ is not trivial
\end{lemm}

\begin{proof}
First observe that for all $\kappa >0$
\begin{equation*}
\mu\big(\,u\in H^{\s}(\T)\;:\; \|u\|_{L^{2}(\T)}\leq \kappa\,\big)={\bf p}\big(\, \om \in \Omega \;:\;\sum_{n\in \Z}\frac{1}{\<n\>^{2}}|g_{n}(\om)|^{2}\leq \kappa^{2}\, \big)>0.
\end{equation*}
Then, by Lemma \ref{ld} and Proposition \ref{prop.cauchy}, the quantities $ \|u\|_{L^{6}(\T)}$ and $f(u)$ are $\mu$ almost surely finite. Hence, the density of $\rho$ does not vanish on a set of positive $\mu$ measure. In other words, $\rho$ is not trivial. 
\end{proof}

\appendix
\section{Appendix}
\subsection{Hamiltonian structure of the transformed form of \eqref{DNLS}}

In this section we give the Hamiltonian structure of the equation related to \eqref{DNLS}.\\
First we define the projection $\Pi$ on the $0$-mean functions :
\begin{equation*}
\Pi\big(f\big)=\sum_{n\in \Z\backslash\{0\}}\a_{n}\e^{inx}, \quad \text{for}\quad f(x)=\sum_{n\in \Z}\a_{n}\e^{inx},
\end{equation*} 
then we introduce the integral operator
\begin{equation*}
\partial^{-1}\,:\,f(x)=\sum_{n\in \Z}\a_{n}\e^{inx} \longmapsto \sum_{n\in \Z\backslash\{0\}}\frac{\a_{n}}{in}\e^{inx}.
\end{equation*} 
Notice that we have 
\begin{equation*}
\partial^{-1}\big(f^{'}\big)=\Pi f =f-\int_{\T}f(x)\,\text{d}x.
\end{equation*} 
 Next we define the operator 
 \begin{equation}\label{def.K}
K(u,v)=\left(\begin{array}{cc} 
-u\partial^{-1}u\cdot & -i+u\partial^{-1}v\cdot \\[2pt]
i+v\partial^{-1}u\cdot & -v\partial^{-1}v\cdot 
\end{array} \right). 
\end{equation}

\begin{lemm} 
For $u,v$, the operator $K(u,v)$ is skew symmetric : $K(u,v)^{*}=-K(u,v)$.
\end{lemm}

\begin{proof}
This is a straightforward computation. We only have to use that $\big(\partial^{-1}\big)^{*}=-\p^{-1}$.
\end{proof}

\noindent Define 
\begin{equation*}
H(u,v)=\int_{\T}\p_{x}u\,\p_{x}v+\frac{3}4i\int_{\T}v^{2}\p_{x}(u^{2})+\frac12\int_{\T}u^{3}v^{3}.
\end{equation*} 
Notice that we also have the expressions
\begin{eqnarray*}
H(u,v)&=&-\int_{\T}\partial^{2}_{x}u\,v+\frac{3}4i\int_{\T}v^{2}\p_{x}(u^{2})+\frac12\int_{\T}u^{3}v^{3}\\
&=&-\int_{\T}u\,\partial^{2}_{x}v-\frac{3}4i\int_{\T}u^{2}\p_{x}(v^{2})+\frac12\int_{\T}u^{3}v^{3},
\end{eqnarray*} 
therefore, we can deduce the variational derivatives
\begin{eqnarray}
\frac{\delta H}{\delta u}(u,v)&=&-\partial^{2}_{x}v-\frac32\,i u\, \p_{x}(v^{2})+\frac32u^{2}v^{3}\label{eq.var1}\\[4pt]
\frac{\delta H}{\delta v}(u,v)&=&-\partial^{2}_{x}u+\frac32\,i v\, \p_{x}(u^{2})+\frac32u^{3}v^{2}.\label{eq.var2}
\end{eqnarray} 
We consider the Hamiltonian system 
\begin{equation}
\left(\begin{array}{c}\label{syst}
\partial_{t}u \\[2pt]
\partial_{t}v
\end{array} \right) =K(u,v)
\left(\begin{array}{c}
\frac{\delta H}{\delta u}(u,v) \\[4pt]
\frac{\delta H}{\delta v}(u,v)
\end{array} \right).
\end{equation}
\noindent Denote by 
\begin{equation*}
F_{u}(t)=2\,\text{Im}\,\int_{\T}u\p_{x}\ov{u}+\frac32\int_{\T}|u|^{4},
\end{equation*} 
and notice that for all $t\in \R$, $F_{u}(t)\in \R$.

\begin{prop}
The system \eqref{syst} is a Hamiltonian formulation of the equation 
\begin{equation}\label{eq.gauge}
i\partial_t u+\p_{x}^{2} u   = i\partial_{x}\big(|u|^{2}u\big)+F_{u}(t)u,
\end{equation}
in the coordinates $(u,v)=(u,\ov{u})$.
\end{prop}
As a consequence, if we set 
\begin{equation}\label{eq.chgt}
v(t,x)=\e^{i\int_{0}^{t}F_{u}(s)\text{d}s}u(t,x),
\end{equation} 
then $v$ is the solution of the equation
\begin{equation}\label{DNLS1}
\left\{
\begin{aligned}
&i\partial_t v+\p_{x}^{2} v   = i\partial_{x}\big(|v|^{2}v\big),\;\;
(t,x)\in\R\times \mathbb{T},\\
&v(0,x)= u_{0}(x).
\end{aligned}
\right.
\end{equation}
Moreover, if $u$ and $v$ are linked by \eqref{eq.chgt}, we have $F_{u}=F_{v}$.
\begin{proof}
We have 
\begin{equation*}
u\,\p^{2}_{x}v=v\,\p^{2}_{x}u+(u\,\p_{x}v)'-(v\,\p_{x}u)',
\end{equation*}
therefore
\begin{equation}\label{eq.ipp1}
\p^{-1}\big(u\,\p^{2}_{x}v\big)=\p^{-1}\big(v\,\p^{2}_{x}u\big)+u\,\p_{x}v-v\,\p_{x}u-\int_{\T}\big(u\,\p_{x}v-v\,\p_{x}u\big).
\end{equation}
Similarly we obtain the relation
\begin{equation}\label{eq.ipp2}
\p^{-1}\big(u^{2}\,\p_{x}(v^{2})\big)=-\p^{-1}\big(v^{2}\,\p_{x}(u^{2})\big)+u^{2}\,v^{2}-\int_{\T}u^{2}\,v^{2}.
\end{equation}
By \eqref{eq.var1}, \eqref{eq.var2}, using \eqref{eq.ipp1} and \eqref{eq.ipp2}, a straightforward computation gives 
\begin{eqnarray*}
\p_{t}u&=&-u\p^{-1}\big(u\,\frac{\delta H}{\delta u}\big)-i\frac{\delta H}{\delta v}+u\p^{-1}\big(v\,\frac{\delta H}{\delta v}\big)\\
&=& i\p^{2}_{x}u+\p_{x}\big(u^{2}\,v\big)-u\int_{\T}(u\,\p_{x}v-v\,\p_{x}u)-\frac3{2}iu\int_{\T}u^{2}v^{2},
\end{eqnarray*}
and 
\begin{eqnarray*}
\p_{t}v&=&i\frac{\delta H}{\delta u}+  v\p^{-1}\big(u\,\frac{\delta H}{\delta u}\big) -v\p^{-1}\big(v\,\frac{\delta H}{\delta v}\big)\\
&=&- i\p^{2}_{x}v+\p_{x}\big(u\,v^{2}\big)- v\int_{\T}(v\,\p_{x}u-u\,\p_{x}v)-\frac3{2}iu\int_{\T}u^{2}v^{2}.
\end{eqnarray*}
Now assume that $v=\ov{u}$. This yields the result, as 
\begin{equation*}
\int_{\T}(u\,\p_{x}\ov{u}-\ov{u}\,\p_{x}u)=2i\,\text{Im}\,\int_{\T}u\,\p_{x}\ov{u}.
\end{equation*}
\end{proof}

\subsection{Invariance of the measure $\rho_{N}$ under a truncated flow of \eqref{eq.gauge}}
We present here a natural finite dimensional approximation of \eqref{eq.gauge} for which $\rho_{N}$ is an invariant measure.

Let $N\geq 1$. Recall that $E_{N}$ is the  the complex vector space $\dis {E_N={\rm span}\Big( (  \e^{inx})_{-N\leq n\leq N}\Big)}$, and that  $\Pi_{N}$ is the spectral   projector  from $L^{2}(\T)$ to  $E_{N}$.

Let $K$ be given by \eqref{def.K}, and consider  the following system 
\begin{equation}
\left(\begin{array}{c}\label{systN}
\partial_{t}u \\[2pt]
\partial_{t}v
\end{array} \right) =\Pi_{N}K(u_{N},v_{N})\Pi_{N}
\left(\begin{array}{c}
\frac{\delta H}{\delta u}(u_{N},v_{N}) \\[4pt]
\frac{\delta H}{\delta v}(u_{N},v_{N})
\end{array} \right).
\end{equation}
This an Hamiltonian system with Hamiltonian $H(\Pi_{N}u,\Pi_{N}v)$. Now we assume that $v=\ov{u}$ and we compute the equation satisfied by $u_{N}$ : this will be a finite dimensional approximation   of \eqref{eq.gauge}. Denote by $\Pi^{\perp}_{N}=1-\Pi_{N}$, then we have  

\begin{lemm}\label{Lem.eqN}
In the coordinates $v_{N}=\ov{u_{N}}$, the system \eqref{systN} reads
\begin{equation}\label{eq.trunc}
i\partial_{t}u+ \partial_{x}^{2}u_{N}=  i\Pi_{N}\Big( \partial_{x}(|u_{N}|^{2}u_{N}  )\Big)+u_{N}F_{u_{N}}(t)  +  R_{N}(u_{N}),
 \end{equation}
where 
\begin{eqnarray*}
R_{N}(u_{N})&=&\frac32  \Pi_{N} \Big(u_{N}\partial^{-1}\Big[ u_{N}\Pi_{N}^{\perp}\big(u_{N}\partial_{x}(\ov{u_{N}}^{2})\big)+\ov{u_{N}}\Pi_{N}^{\perp}\big(\ov{u_{N}}\partial_{x}({u_{N}}^{2})\big)\Big]\Big)\\
&&+\frac32i  \Pi_{N} \Big( u_{N}\partial^{-1}\Big[ {u_{N}}\Pi^{\perp}_{N}\big(|u_{N}|^{4}\ov{u_{N}}\big)-\ov{u_{N}}\Pi^{\perp}_{N}\big(|u_{N}|^{4}{u_{N}}\big) \Big] \Big).
\end{eqnarray*}
\end{lemm}

\begin{proof}
The  proof is   a direct computation. By \eqref{systN}, the equation on $u_{N}$ reads 
\begin{equation}\label{eq1}
\partial_{t}u=\Pi_{N}\Big( -u_{N}\partial^{-1}(u_{N}f_{N})-i\ov{f_{N}}+u_{N}\partial^{-1}(\ov{u_{N}f_{N}})\Big),
\end{equation}
where 
\begin{equation*}
f_{N}=\Pi_{N}\Big( -\partial^{2}_{x} \ov{u_{N}}  -\frac32i u_{N} \partial_{x}(\ov{u_{N}}^{2})+\frac32 |u_{N}|^{4}\ov{u_{N}}\Big).
\end{equation*}
Thanks to \eqref{eq.ipp1}  we deduce from \eqref{eq1} that
\begin{multline}\label{eqq}
\partial_{t}u=i \partial_{x}^{2}u_{N}+\frac32 \Pi_{N}\Big(\ov{u_{N}}\partial_{x}(u^{2}_{N}  )\Big)-\frac32 i\Pi_{N}\Big(   |u_{N}|^{4} {u_{N}})\Big) +\\
\begin{aligned}
&+  \Pi_{N}\Big(u_{N}^{2}\partial_{x}\ov{u_{N}}-|u_{N}|^{2}\partial_{x}u_{N}   \Big)-u_{N}\int_{\T}(u_{N}\partial_{x}\ov{u_{N}}-\ov{u_{N}}\partial_{x}{u_{N}})+\\
&+\frac32 i \Pi_{N} \Big( u_{N}\partial^{-1}\Big[ u_{N}\Pi_{N}\big(u_{N}\partial_{x}(\ov{u_{N}}^{2})\big)+\ov{u_{N}}\Pi_{N}\big(\ov{u_{N}}\partial_{x}({u_{N}}^{2})\big)\Big] \Big)+\\
&+\frac32  \Pi_{N} \Big( u_{N}\partial^{-1}\Big[ \ov{u_{N}}\Pi_{N}\big(|u_{N}|^{4}{u_{N}}\big)- {u_{N}}\Pi_{N}\big(|u_{N}|^{4}\ov{u_{N}}\big)\big)\Big] \Big).
\end{aligned}
\end{multline}
Using \eqref{eq.ipp2} we obtain, with $\Pi_{N}^{\perp}=1-\Pi_{N}$
\begin{multline}\label{Nn1}
\partial^{-1}\Big[ u_{N}\Pi_{N}\big(u_{N}\partial_{x}(\ov{u_{N}}^{2})\big)+\ov{u_{N}}\Pi_{N}\big(\ov{u_{N}}\partial_{x}({u_{N}}^{2})\big)\Big] =\\
\begin{aligned}
&=-\partial^{-1}\Big[ u_{N}\Pi_{N}^{\perp}\big(u_{N}\partial_{x}(\ov{u_{N}}^{2})\big)+\ov{u_{N}}\Pi_{N}^{\perp}\big(\ov{u_{N}}\partial_{x}({u_{N}}^{2})\big)\Big]   +\partial^{-1}\Big[u_{N}^{2}\partial_{x}(\ov{u_{N}}^{2})+\ov{u_{N}}^{2}\partial_{x}({u_{N}}^{2})\Big]  \\
&=-\partial^{-1}\Big[ u_{N}\Pi_{N}^{\perp}\big(u_{N}\partial_{x}(\ov{u_{N}}^{2})\big)+\ov{u_{N}}\Pi_{N}^{\perp}\big(\ov{u_{N}}\partial_{x}({u_{N}}^{2})\big)\Big]   +|u_{N}|^{4} -\int_{\T}|u_{N}|^{4}.    \\
 \end{aligned}
\end{multline}
We can also write 
\begin{multline}\label{Nn2}
 \ov{u_{N}}\Pi_{N}\big(|u_{N}|^{4}{u_{N}}\big)- {u_{N}}\Pi_{N}\big(|u_{N}|^{4}\ov{u_{N}}\big)\big)  =\\
 -\ov{u_{N}}\Pi^{\perp}_{N}\big(|u_{N}|^{4}{u_{N}}\big)+ {u_{N}}\Pi^{\perp}_{N}\big(|u_{N}|^{4}\ov{u_{N}}\big)\big).
 \end{multline}
Thus, by \eqref{Nn1} and \eqref{Nn2}, equation \eqref{eqq} becomes
\begin{eqnarray*}
\partial_{t}u&=& i \partial_{x}^{2}u_{N}+  \Pi_{N}\Big( \partial_{x}(|u_{N}|^{2}u_{N}  )\Big)-iu_{N}F_{u_{N}}(t)  +   \nonumber \\
&&-\frac32i  \Pi_{N} \Big(u_{N}\partial^{-1}\Big[ u_{N}\Pi_{N}^{\perp}\big(u_{N}\partial_{x}(\ov{u_{N}}^{2})\big)+\ov{u_{N}}\Pi_{N}^{\perp}\big(\ov{u_{N}}\partial_{x}({u_{N}}^{2})\big)\Big]\Big)\\
&&+\frac32  \Pi_{N} \Big( u_{N}\partial^{-1}\Big[ {u_{N}}\Pi^{\perp}_{N}\big(|u_{N}|^{4}\ov{u_{N}}\big)\big)-\ov{u_{N}}\Pi^{\perp}_{N}\big(|u_{N}|^{4}{u_{N}}\big) \Big] \Big),
\end{eqnarray*}
which is the claim.
\end{proof}

In the sequel we fix $\s<\frac12$, and we consider \eqref{ODEN} as a Cauchy problem with initial condition in  $H^{\s}(\T)$
 \begin{equation}\label{ODE}
\left\{
\begin{aligned}
&i\partial_{t}u+ \partial_{x}^{2}u_{N}=  i\Pi_{N}\Big( \partial_{x}(|u_{N}|^{2}u_{N}  )\Big)+u_{N}F_{u_{N}}(t)  +  R_{N}(u_{N}),\;\;
(t,x)\in\R\times \mathbb{T},\\
&u(0,x)=  u_{0}(x) \in H^{\s}(\T).
\end{aligned}
\right.
\end{equation}
We now state the main result of this section.
 \begin{prop}\label{Prop.inv}
The equation \eqref{ODE} has a well-defined global flow $\Phi_{N}$. Moreover, the measure $\rho_{N}$ is invariant under  $\Phi_{N}$  : For any Borel set $A\subset H^{\s}(\T)$ and for all $t\in \R$, $\rho_{N}\big(\Phi_{N}(t)(A)\big)=\rho_{N}\big( A\big)$.
\end{prop}
For the proof of Proposition \ref{Prop.inv}, we first need the following result 
\begin{lemm}
The equation 
\begin{equation}\label{ODEN}
\left\{
\begin{aligned}
&i\partial_{t}u+ \partial_{x}^{2}u_{N}=  i\Pi_{N}\Big( \partial_{x}(|u_{N}|^{2}u_{N}  )\Big)+u_{N}F_{u_{N}}(t)  +  R_{N}(u_{N}),\;\;
(t,x)\in\R\times \mathbb{T},\\
&u(0,x)=\Pi_{N}\big( u_{0}(x)\big)\in E_{N}.
\end{aligned}
\right.
\end{equation}
is an Hamiltonian ODE. Moreover, the mass $\|u(t)\|_{L^{2}(\T)}$ is conserved under the flow of \eqref{ODEN}. As a consequence, \eqref{ODEN} has a well-defined global flow $\wt{\Phi}_{N}$.
\end{lemm}

\begin{proof}
The first statement is clear by the previous construction. We now check that the $L^{2}-$norm of $u$ is conserved. Multiply \eqref{eq.trunc} with $\ov{u}$, integrate over $x\in \T$ and take the imaginary part. In the sequel we use that $\Pi_{N}^{2}=\Pi_{N}$ and $\Pi_{N}^{*}=\Pi_{N}$. Firstly by integration by parts, 
\begin{equation}\label{Ilk1}
\int_{\T}\ov{u} \,\partial_{x}^{2}u_{N}=\int_{\T}\ov{u_{N}} \,\partial_{x}^{2}u_{N}=-\int_{\T}|\partial_{x}u_{N}|^{2}\in \R.
\end{equation} 
Then 
\begin{eqnarray}
\text{Im}\;\int_{\T}i \ov{u} \,\Pi_{N}\Big( \partial_{x}(|u_{N}|^{2}u_{N}  )\Big)&=& \text{Re}\;\int_{\T} \ov{u_{N}}   \partial_{x}(|u_{N}|^{2}u_{N}  ) \nonumber\\
&=&-\text{Re}\;\int_{\T} (\partial_{x}\ov{u_{N}})    |u_{N}|^{2}u_{N}\nonumber\\
&=&-\frac14 \int_{\T}\partial_{x}(|u_{N}|^{4})=0\label{Ilk2}.
\end{eqnarray}
Now observe that if $f$ is real-valued, then $\partial^{-1}f$ is also real valued. Then it is easy to see that  
\begin{equation}\label{Ilk3}
\int_{\T} \ov{u} \,R_{N}(u_{N})=\int_{\T} \ov{u_{N}}\, R_{N}(u_{N})\in \R.
\end{equation}
Finally by \eqref{Ilk1}, \eqref{Ilk2} and \eqref{Ilk3} we obtain that $\dis \frac{d}{d t}\|u(t)\|^{2}_{L^{2}(\T)}=0$ which yields the result.
\end{proof}
Recall the definitions \eqref{def.mu} of $\mu_{N}$ and \eqref{G.N} of $G_{N}$. Then we define the measure $\wt{\rho}_{N}$ on $E_{N}$ by 
\begin{equation*}
\text{d}\wt{\rho}_{N}(u)=G_{N}(u)\text{d}\mu_{N}(u).
\end{equation*}
Then we have 
\begin{lemm}\label{lem.invN}
The measure $\widetilde{\rho}_{N}$ is invariant under the flow $\wt{\Phi}_{N}$ of \eqref{ODEN}.
\end{lemm}
 
 \begin{proof}
 The proof is a direct application of the Liouville thereom. See e.g. \cite[Section 8]{BTT} for a similar argument.
 \end{proof}
 
 \begin{proof}[Proof of Proposition \ref{Prop.inv}]
We decompose the space $H^{\s}(\T)=E_{N}^{\perp}\oplus E_{N}$. From the previous analysis, we observe that the flow $\Phi_{N}$ of \eqref{ODE} is given by  $\Phi_{N}=\big(Id, \wt{\Phi}_{N}\big)$. Finally, the invariance of $\rho_{N}$ follows from Lemma \ref{lem.invN} and invariance of the Gaussian measure under the trivial flow on the high frequency part.
\end{proof}


\end{document}